\documentclass[11pt]{amsart}
\usepackage{amsmath}
\usepackage{amssymb}
\usepackage{amsthm}
\usepackage{latexsym}

\def\NZQ{\mathbb}               
\def\NN{{\NZQ N}}

\def\ZZ{{\NZQ Z}}

\def\CC{{\NZQ C}}
\def\F2{{\NZQ F}_2}

\def\opn#1#2{\def#1{\operatorname{#2}}} 
\opn\chara{char} \opn\length{\ell} \opn\pd{pd} \opn\rk{rk}
\opn\projdim{proj\,dim} \opn\injdim{inj\,dim} \opn\rank{rank}
\opn\depth{depth} \opn\codepth{codepth} \opn\grade{grade}
\opn\height{height} \opn\embdim{emb\,dim} \opn\codim{codim}

\opn\Tr{Tr} \opn\bigrank{big\,rank}
\opn\superheight{superheight}\opn\lcm{lcm}
\opn\trdeg{tr\,deg}%
\opn\reg{reg} \opn\lreg{lreg} \opn\skel{skel}
\opn\Gr{Gr}
\opn\ann{ann}
\opn\sign{sign}

\opn\div{div} \opn\Div{Div} \opn\cl{cl} \opn\Cl{Cl}
\opn\Spec{Spec} \opn\Supp{Supp} \opn\supp{supp} \opn\Sing{Sing}
\opn\Ass{Ass}\opn\fdepth{fdepth}
\opn\Ann{Ann} \opn\Rad{Rad} \opn\Soc{Soc}
\opn\Sym{Sym} \opn\Ker{Ker} \opn\Coker{Coker} \opn\Im{Im}
\opn\Hom{Hom} \opn\Tor{Tor} \opn\Ext{Ext} \opn\End{End}
\opn\Aut{Aut} \opn\id{id} \opn\ini{in} \opn\tr{tr}

\opn\nat{nat}\opn\it{it}
\opn\pff{proof}
\opn\Pf{proof} \opn\GL{GL} \opn\SL{SL} \opn\mod{mod} \opn\ord{ord}
\opn\aff{aff} \opn\con{conv} \opn\relint{relint} \opn\st{st}
\opn\lk{lk} \opn\cn{cn} \opn\core{core} \opn\vol{vol}
\opn\link{link} \opn\star{star} \opn\skel{skel} \opn\indeg{indeg}
\opn\Ass{Ass} \opn\Min{Min} \opn\sdepth{sdepth} \opn\depth{depth}
\opn\gr{gr}

\def\pot#1#2{#1[\kern-0.28ex[#2]\kern-0.28ex]}

\opn\dirlim{\underrightarrow{\lim}}
\opn\inivlim{\underleftarrow{\lim}}
\let\iso=\cong

\let\to=\rightarrow

\def\Implies{\ifmmode\Longrightarrow \else
     \unskip${}\Longrightarrow{}$\ignorespaces\fi}
\def\implies{\ifmmode\Rightarrow \else
     \unskip${}\Rightarrow{}$\ignorespaces\fi}
\def\iff{\ifmmode\Longleftrightarrow \else
     \unskip${}\Longleftrightarrow{}$\ignorespaces\fi}

\let\:=\colon
\opn\d{d}
\newtheorem{Theorem}{Theorem}[section]
\newtheorem{Lemma}[Theorem]{Lemma}
\newtheorem{Corollary}[Theorem]{Corollary}

\newtheorem{Remark}[Theorem]{Remark}

\newtheorem{Example}[Theorem]{Example}

\newtheorem{Definition}[Theorem]{Definition}

\let\epsilon\varepsilon
\let\phi=\varphi
\let\kappa=\varkappa
\def\qed{\ifhmode\textqed\fi
   \ifmmode\ifinner\quad\qedsymbol\else\dispqed\fi\fi}
\def\textqed{\unskip\nobreak\penalty50
    \hskip2em\hbox{}\nobreak\hfil\qedsymbol
    \parfillskip=0pt \finalhyphendemerits=0}
\def\dispqed{\rlap{\qquad\qedsymbol}}

\opn\Gin{Gin}

\def\CC{{\mathcal C}}

\def\FF{{\mathcal F}}

\newcommand{\ri}{\mathrm{ri}}
\newcommand{\HP}{\mathrm{HP}}

\opn\inii{in} \opn\inim{inm} \opn\rate{rate}

\numberwithin{equation}{section}

\title{Some algebraic invariants of edge ideal of circulant graphs}


\author{Giancarlo Rinaldo}
\address{Department of Mathematics\\
University of Trento\\
via Sommarive, 14\\
38123 Povo (Trento), Italy
}

\date{}

\begin{document}
\maketitle

\begin{abstract}
Let $G$ be the circulant graph $C_n(S)$ with $S\subseteq\{ 1,\ldots,\left \lfloor\frac{n}{2}\right \rfloor\}$ and let $I(G)$ be its edge ideal in the ring $K[x_0,\ldots,x_{n-1}]$. Under the hypothesis that $n$ is prime we : 1) compute the regularity index of $R/I(G)$; 2) compute the Castelnuovo-Mumford regularity when $R/I(G)$ is Cohen-Macaulay; 3) prove that the circulant graphs with $S=\{1,\ldots,s\}$ are sequentially $S_2$ .
We end characterizing the Cohen-Macaulay circulant graphs of Krull dimension $2$ and computing their  Cohen-Macaulay type and Castelnuovo-Mumford regularity. 
\end{abstract}

\section*{Introduction}\label{sec:intro}
Let $S\subseteq\{ 1,2,\ldots,\left \lfloor\frac{n}{2}\right \rfloor\}$. The \textit{circulant graph} $G:=C_n(S)$ is a graph with vertex set $\ZZ_n=\{0,\ldots,n-1\}$ and edge set $E(G) := \{ \{i, j\} \mid |j-i|_n \in S \}$ where $|k|_n=\min\{|k|,n-|k|\}$.

Let $R= K[x_0, \dots, x_{n-1}]$ be the polynomial ring on $n$ variables over a field $K$. The \textit{edge ideal} of $G$, denoted by $I(G)$, is the ideal of $R$ generated by all square-free monomials $x_i x_j$ such that $\{i,j\} \in E(G)$. Edge ideals of a graph have been introduced by Villarreal [11] in 1990, where he studied the Cohen--Macaulay property of such ideals. Many authors have focused their attention on such ideals (see \cite{HH}, \cite{FT}).  A known fact about Cohen-Macaulay edge ideals is that they are well-covered.


A graph $G$ is said \textit{well-covered} if all the maximal independent sets of $G$ have the same cardinality. Recently well-covered circulant graphs have been studied (see \cite{BH0}, \cite{BH1}, \cite{Ho}).  In \cite{MTW} and \cite{EMT} the authors studied well-covered circulant graphs that are Cohen-Macaulay.

In this article we put in relation the values $n$, $S$ of a circulant graph $C_n(S)$  and algebraic invariants of $R/I(G)$.  In particular we study the regularity index, the Castelnuovo-Mumford regularity, the Cohen-Macaulayness and Serre's condition of $R/I(G)$.

In the first section we recall some concepts and notations and preliminary notions.

In the second section under the hypothesis that $n$ is prime we observe that the regularity index of $R/I(G)$ is $1$ obtaining as a by-product the Castelnuovo-Mumford regularity of the ring when it is Cohen-Macaulay. 

In the third section we prove that each $k$-skeleton of the simplicial complex of the independent set of $G=C_n(S)$ is connected when $n$ is prime. As an application we prove that the circulant graphs $C_n(\{1,\ldots,s\})$ (studied in \cite{BH0}, \cite{BH1}, \cite{EMT}, \cite{Ho}, \cite{Mo},\cite{MTW}) are sequentially $S_2$ (see \cite{HTYZ}).

In the last section we characterize the Cohen-Macaulay circulant graphs of Krull dimension $2$ and compute their  Cohen-Macaulay type and Castelnuovo--Mumford regularity. 

\section{Preliminaries}\label{sec:pre}
In this section we recall some concepts and notations on graphs and on simplicial complexes that we will use in the article. Let $G$ be a simple graph with vertex set $V(G)$ and the edge set $E(G)$. A subset $C$ of $V(G)$ is called a \textit{clique} of $G$ if for all $i$ and $j$ belonging to $C$ with $i \neq j$ one has $\{i, j\} \in E(G)$. A subset $A$ of $V(G)$ is called an \textit{independent set} of $G$ if no two vertices of $A$ are adjacent.The \textit{complement graph} $\bar{G}$ of $G$ is the graph with vertex set $V(\bar{G})=V(G)$ and edge set $E(\bar{G})= \{\{u,v\}\in V(G)^2\mid \{u,v\}\notin E(G)\}$.

Set $V = \{x_1, \ldots, x_n\}$. A \textit{simplicial complex} $\Delta$ on the vertex set $V$ is a collection of subsets of $V$ such that
\begin{enumerate}
\item[(i)] $\{x_i\} \in \Delta$  for all $x_i \in V$;
\item[(ii)] $F \in \Delta$ and $G\subseteq F$ imply $G \in \Delta$.
\end{enumerate}
An element $F \in \Delta$ is called a \textit{face} of $\Delta$. A maximal face of $\Delta$  with respect to inclusion is called a \textit{facet} of $\Delta$.

If $\Delta$ is a simplicial complex  with facets $F_1, \ldots, F_q$, we call $\{F_1, \ldots, F_q\}$ the facet set of $\Delta$ and we denote it by $\FF(\Delta)$.
The dimension of a face $F \in \Delta$ is $\dim F = |F|-1$, and the dimension of $\Delta$ is the maximum of the dimensions of all facets in $\FF(\Delta)$. If all facets of $\Delta$ have the same dimension, then $\Delta$ is called \textit{pure}. Let $d-1$ the dimension of $\Delta$ and let $f_i$ be the number of faces of $\Delta$ of dimension $i$ with the convention that $f_{-1}=1$. Then the $f$-vector of $\Delta$ is the $d$-tuple $f(\Delta)=(f_{-1},f_0,\ldots,f_{d-1})$. The $h$-vector of $\Delta$ is $h(\Delta)=(h_0,h_1,\ldots,h_d)$ with
\[
 h_k=\sum_{i=0}^{k}(-1)^{k-i}\binom{d-i}{k-i} f_{i-1}.
\]
The sum
\[
 \widetilde{\chi}(\Delta)=\sum_{i=-1}^{d-1}(-1)^{i}f_i
\]
is called the reduced Euler characteristic of $\Delta$ and $h_d=(-1)^{d-1}\widetilde{\chi}(\Delta)$. Given any simplicial complex $\Delta$ on $V$, we can associate a monomial ideal $I_\Delta$ in the polynomial ring $R$ as follows:
\[
 I_\Delta=(\{x_{j_1}x_{j_2}\cdots x_{j_r}: \{x_{j_1},x_{j_2},\ldots,x_{j_r}\}\notin \Delta\}).
\]
$R/I_\Delta$ is called Stanley-Reisner ring and its Krull dimension is $d$. If $G$ is a graph we call the \textit{independent complex} of $G$ by
\[
\Delta(G)=\{A\subset V(G): A \mbox{ is an independent set of }G\}. 
\]
The \textit{clique complex} of a graph $G$ is the simplicial complex whose faces are the cliques of $G$.
Let $\mathbb{F}$ be the minimal free resolution of the quotient ring $R/I(G)$.
Then
\[\mathbb{F}:\  0\to F_p \to \cdots \to F_{p-1}\to\cdots\to F_0\to R/I(G)\rightarrow 0\]
with $F_i=\mathop\oplus\limits_j R(-j)^{\beta_{ij}}$. The numbers $\beta_{ij}$ are  called the Betti numbers of $\mathbb{F}$. The Castelnuovo--Mumford regularity of $R/I(G)$, denoted by $\reg R/I(G)$, is defined by
\[
\reg  R/I(G)  = \max \{ j - i : \beta_{ij}\neq 0 \}.
\]
A graph $G$ is said Cohen-Macaulay if the ring $R/I(G)$, or equivalentelly $R/I_{\Delta(G)}$ is Cohen-Macaulay (over the field $K$) (see \cite{BH2}, \cite{MS}, \cite{Vi0}).
The Cohen-Macaulay type of $R/I(G)$ is equal to the last total Betti number in the minimal free  resolution $\mathbb{F}$.

We end this section with the following
\begin{Remark}\label{cliqueindependent}
Let $T=\{1,2,\ldots,\left \lfloor\frac{n}{2}\right \rfloor\}$ and $G$ be a circulant graph on $S\subseteq T$ with $s=|S|$, then:
\begin{enumerate}
 \item $\bar{G}$ is a circulant graph on $\bar{S}=T\setminus S$;
 \item The clique complex of $\bar{G}$ is the independent complex of $G$, $\Delta(G)$;
 \item 
\[ 
|E(G)|=
\left \{
  \begin{tabular}{cl}
  $ns-\frac{n}{2}$ & if $n$ is even and $\frac{n}{2}\in S$\\
  $ns$ & otherwise.  
  \end{tabular}
\right. 
\]
\end{enumerate}
\end{Remark}

\section{Regularity and connectedness of the independent complex of circulant graphs of prime order}

We recall some basic facts about the regularity index (see also \cite{Va}). Let $R$ be standard graded ring and $I$ be a homogeneous ideal. The \textit{Hilbert function} $H_{R/I} : \mathbb{N} \rightarrow \mathbb{N}$ is defined by 
\[
H_{R/I} (k) := \dim_K (R/I)_k
\]
and the Hilbert-Poincar\'e series of $R/I$ is given by
\[
\HP_{R/I} (t) := \sum_{k \in \NN} H_{R/I}(k) t^k. 
\]
By Hilbert-Serre theorem, the Hilbert-Poincar\'e series of $R/I$ is a rational function, that is
\[
\HP_{R/I}(t) = \frac{h(t)}{(1-t)^n}.
\]
There exists a unique polynomial such that $H_{R/I}(k) = P_{R/I}(k)$ for all $k\gg 0$. 
The minimum integer $k_0 \in \NN$ such that $H_{R/I}(k) = P_{R/I}(k) \; \forall \ k \geq k_0$ is called \textit{regularity index} and we denote it by $\ri(R/I)$.  
\begin{Remark}\label{remri}
 Let $R/I_\Delta$ be a Stanley-Reisner ring. Then 
 \[ 
\ri(R/I_\Delta)=
\left \{
  \begin{tabular}{cl}
  $0$ & if $h_d=0$  \\
  $1$ & if $h_d\neq 0$  
  \end{tabular}
\right. 
\]
\end{Remark}
\begin{proof}
By the hypothesis the Hilbert series can be represented by the reduced rational function
 \[
 \frac{h(t)}{(1-t)^d}  
 \]
 where $d$ is the Krull dimension of $R/I_\Delta$ and $h(t)=\sum_{i=0}^{d} h_i t^i$ where $h_i$ are the entries of the $h$-vector of $\Delta$. We observe that $\ri(R/I)=\max(0, \deg h(t)-d+1)$. If $\ri(R/I_\Delta)>0$ then $\deg h(t)>d-1$. But since $\deg h(t)\leq d$ we have $\deg h(t)=d$. Therefore $h_d\neq 0$ and $\ri(R/I_\Delta)=1$. The other case follows by the same argument.
\end{proof}

\begin{Lemma}\label{primeFvector}
Let $G$ be a circulant graph on $S$ with $n$ prime. Then the entries of the $f$-vector of $\Delta(G)$ are
\[
f_i=n f'_i 
\]
with $0\leq i\leq d-1$ and $f'_i=f_{i,0}/(i+1)\in \NN$ where $f_{i,0}$ is the number of faces of dimension $i$ containing the vertex $0$.
\end{Lemma}

\begin{proof}
Call $\FF_i\subset \Delta$ the set of faces of dimension $i$, that is 
\[
\FF_i=\{F_1,\ldots,F_{f_i}\}. 
\]
Let $f_{i,j}$, number of faces in $\FF_i$  containing a given vertex $j=0,\ldots, n-1$. Since $G$ is circulant
\[
 f_{i,j}=f_{i,0} \mbox{ for all }j\in \{0,\ldots,n-1\}.
\]
Let $A\in \F2^{f_i\times n}=(a_{jk})$ be the incidence matrix with $a_{jk}=1$ if the vertex $k-1$ belongs to the facet $F_j$ and $0$ otherwise. We observe that each row has exactly $i+1$ $1$-entries. Hence summing the entries of the matrix we have $(i+1) f_i$. Moreover each column has exactly $f_{i,j}$ non zero entries. That is
\[
n f_{i,0}= (i+1) f_i.
\]
Since $n$ is prime the assertion follows.
\end{proof}

\begin{Theorem}\label{riprime}
 Let $G$ be a circulant graph on $S$ with $n$ prime. Then
 \[
  \ri(R/I(G))=1.
 \]
\end{Theorem}
\begin{proof}
  By Remark \ref{remri} it is sufficient to show that $h_d$ is different from $0$. Since 
  \[
   |h_d|=|\sum_{i=0}^{d} (-1)^i f_{i-1}|\neq 0,
  \]
it is sufficient to show that the reduced Euler formula is different from $0$, that is
\[
 \sum_{i=1}^d (-1)^i f_{i-1}\neq 1. 
\]
By Lemma \ref{primeFvector} we obtain
\[
 \sum_{i=1}^d (-1)^i f_{i-1}=n\sum_{i=1}^d (-1)^i f_{i-1}'
\]
since $n$  is prime and the assertion follows.  
\end{proof}

\begin{Remark}
 In the proof of Theorem \ref{riprime} we are giving a partial positive answer to the Conjecture 5.38 of \cite{Ho} that states that for all circulant graphs 
 $\widetilde{\chi}(\Delta)\neq 0$. In the article \cite{RR} we found other families of circulant graphs satisfying the previous property. In the same article we found a counterexample that disprove the conjecture in general.
\end{Remark}

\begin{Corollary}\label{RegEqualDepth}
 Let $G$ be a circulant graph on $S$ with $n$ prime that is Cohen-Macaulay. Then $\reg R/I(G)=\depth R/I(G)$.
\end{Corollary}
\begin{proof}
 By Corollary 4.8 of \cite{Ei} since $\ri(R/I)=1$ the assertion follows.
\end{proof}

\section{Sequentially $S_2$ circulant graphs of prime order and connectedness}
In this section we study good properties of the independent complex $\Delta(G)$ of a circulant graph $G$ that have prime order. We start by the following
\begin{Definition}
Let $\Delta$ be a simplicial complex then we define the pure simplicial complexes $\Delta^{[k]}$  whose facets are
\[
 \FF(\Delta^{[k]})=\{F\in\Delta:\dim(F)=k\},\hspace{1cm} 0\leq k\leq \dim(\Delta). 
\]
\end{Definition}

One interesting property of Cohen-Macaulay ring $R/I_\Delta$ is that the each simplicial complex $\Delta^{[k]}$ is connected. Hence the following Lemma is of interest.
\begin{Lemma}\label{skeleton}
Let $G$ be a circulant graph on $S$ with $n$ prime. Then the $k$-skeleton of the simplicial complex $\Delta$, $\Delta^{[k]}$ is connected for every $k\geq 1$.
\end{Lemma}
\begin{proof}
To prove the claim we find a Hamiltonian cycle connecting all the vertices in $V=\{0,\ldots,n-1\}$ of the $1$-skeleton of $\Delta^{[k]}$. Then it follows that since the $1$-skeleton is connected then $\Delta^{[k]}$ is connected, too.

We assume without loss of generality that $F_0=\{v_0, v_1,\ldots, v_{k}\}\in \Delta^{[k]}$ such that $v_0=0$, $v_1=s\in S$. We define the set
\[
F_j=\{v_{0,j}, v_{1,j},\ldots, v_{k,j}\}
\]
with $v_{i,j}=v_i+js \mod n$. It is easy to observe that since $F_0$ is in $\Delta^{[k]}$  and $G$ is circulant, $F_j$ is in $\Delta^{[k]}$, too.

Moreover if we focus on the first two vertices of $F_j$ we obtain that
\[
 v_{1,j}=v_{0,j-1}\mbox{ for all }j=1,\ldots,n-1,
\]
and $v_{0,n-1}=v_{1,0}$. Since
\[
 v_{0,j}=js \mod n
\]
the set $\{v_0,\ldots, v_{n-1}\}$, by the primality of $n$, is equal to $V$. Hence the cycle with vertices
\[
 v_{0,0},v_{0,1},\ldots, v_{0,n-1}
\]
and edges
\[
 \{v_{0,0},v_{0,1}\},\ldots, \{v_{0,n-2}, v_{0,n-1}\},\{v_{0,n-1}, v_{0,0}\}
\]
is a Hamiltonian cycle and the assertion follows.
\end{proof}

Recall that a finitely generated graded module $M$ over a Noetherian graded $K$-algebra $R$ is said to satisfy the Serre's condition $S_r$ if
\[
 \depth M_\mathfrak{p}\geq \min(r,\dim M_ \mathfrak{p}),
\]
for all $\mathfrak{p}\in \Spec(R)$.

\begin{Definition}
Let $M$ be a finitely generated $\ZZ$-graded module over a standard graded $K$-algebra $R$ where $K$ is a field. For a positive integer $r$ we say that $M$ is {\em sequentially $S_r$} if there exists a finite filtration of graded $R$-modules
\[
 0=M_0\subset M_1\subset \cdots\subset M_r=M
\]
such that each $M_i/M_{i-1}$ satisfies th $S_r$ condition and the Krull dimensions of the quotients are increasing:
\[
 \dim(M_1/M_0)<\dim(M_2/M_1)< \cdots<\dim(M_{t}/M_{t-1}).
\]
\end{Definition}

A nice characterization of sequentially $S_2$ simplicial complexes is the following:

\begin{Theorem}[\cite{HTYZ}]\label{th:terai}
 Let $\Delta$ be a simplicial complex with vertex set $V$. Then $\Delta$ is sequentially $S_2$ if and only if the following conditions hold:
 \begin{enumerate}
  \item $\Delta^{[i]}$ is connected for all $i\geq 1$;
  \item $\link_\Delta(x)$ is sequentially $S_2$ for all $x\in V$.
 \end{enumerate}
\end{Theorem}

\begin{Example}\label{ex:O_G}
Let $G$ be the circulant graph $C_6(\{1\})$. Then its simplicial complex $\Delta$ is connected, but  $\Delta^{[2]}$ is not (see Figure \ref{C6}).
\begin{center}
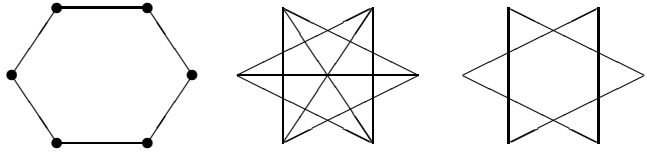
\begin{figure}

\setlength{\unitlength}{.3cm}
\begin{picture}(28,8)

\put(02,00){\circle*{.5}}
\put(06,00){\circle*{.5}}

\put(00,03){\circle*{.5}}
\put(08,03){\circle*{.5}}

\put(02,06){\circle*{.5}}
\put(06,06){\circle*{.5}}

\put(02,00){\line(1,0){4}}
\put(02,06){\line(1,0){4}}

\put(00,03){\line(2,-3){2}}
\put(00,03){\line(2,3){2}}

\put(08,03){\line(-2,-3){2}}
\put(08,03){\line(-2,3){2}}




\put(12,00){\line(2,3){4}}
\put(12,06){\line(2,-3){4}}

\put(10,03){\line(2,1){6}}
\put(10,03){\line(1,0){8}}
\put(10,03){\line(2,-1){6}}

\put(18,03){\line(-2,-1){6}}
\put(18,03){\line(-2,1){6}}

\put(12,00){\line(0,1){6}}
\put(16,00){\line(0,1){6}}




\put(20,03){\line(2,1){6}}
\put(20,03){\line(2,-1){6}}

\put(28,03){\line(-2,-1){6}}
\put(28,03){\line(-2,1){6}}

\put(22,00){\line(0,1){6}}
\put(26,00){\line(0,1){6}}
\end{picture}
\caption{$G=C_6(\{1\}$, $\Delta$ and $\Delta^{[2]}$.}\label{C6}  
\end{figure}
 
\end{center}

\end{Example}

Sequentially Cohen-Macaulay cycles have been characterized in \cite{FT}, that are in our notation are just $C_3(\{1\})$ and $C_5(\{1\})$. In \cite{HTYZ} the authors proved that the only  sequentially $S_2$ are the odd cycles. The following is related to these results.
\begin{Theorem}
 Let $G$ be the circulant graph $C_n(\{1,\ldots,s\})$ with $n$ prime. Then $G$ is sequentially $S_2$.
\end{Theorem}
\begin{proof}
By Lemma \ref{skeleton} the first condition of Theorem \ref{th:terai} is satisfied. To check the second condition of Theorem \ref{th:terai} we prove that $K[\link_\Delta(x_0)]$, is sequentially Cohen Macaulay. We observe that
\[
K[\link_\Delta(x_0)]\cong (R/I(G))_{x_0}\cong  K[x_0^{\pm 1}][x_1,\ldots,x_{n-1}]/I(G)'
\]
where $I(G)'$ is obtained by the $K$-algebra homomorphism induced by the mapping $x_0\to 1$. Since the vertices  adjacent to  $0$ are $\{1,\ldots,s\}\cup\{n-s,\ldots,n-1\}$ we have that
\[
 I(G)'=I(G')+(x_1,\ldots,x_s)+(x_{n_s},\ldots,x_{n-1}).
\]
with $G'$ be the subgraph of $G$ induced by the vertices $\{s+1,\ldots,n-(s+1)\}$. That is
\[
(R/I(G))_{x_0}\cong K[x_{s+1},\ldots,x_{n-s-1}]/I(G').
\]
We claim that $I(G')$ is chordal, hence it is sequentially Cohen-Macaulay by Theorem 3.2 of \cite{FT}. To prove the claim we observe that the labelling on the vertices of $G'$
\[
 s+1,s+2,\ldots, n-s-1
\]
induces a perfect elimination ordering, that is $N^+(i)=\{j: \{i,j\}\in E(G'), i<j \}$ is a clique. Let $j$, $k\in N^+(i)$. That is $\{i,j\}$ and $\{i,k\}$ are two edges with $i<j$ and $i<k$ and assume $j<k$. Then $|j-i|_n=j-i\leq s$ and $|k-i|_n=k-i\leq s$. Moreover
\[
 0<j-i<k-i\leq s.
\]
Hence it follows $|k-j|=k-j<s$. Therefore $\{j,k\}\in E(G')$ and  $N^+(i)$ is a clique.
\end{proof}

\begin{Example}
If a ring is Cohen-Macaulay it is pure and sequentially $S_n$ for all $n$. The circulant graph of prime order with minimum number of vertices that is Cohen-Macaulay and has Krull dimension greater than $2$ is $C_{13}(\{1,5\})$ (see \cite{EMT}). 
\end{Example}

\section{Cohen-Macaulay circulant graphs of dimension $2$ and their Castelnuovo-Mumford regularity}
We start this section by the following
\begin{Theorem}\label{CM1}
 Let $G$ be the circulant graph $C_n(S)$ with $S\subset \{1,\ldots,\left \lfloor\frac{n}{2}\right \rfloor\}$. The following conditions are equivalent:
  \begin{enumerate}
  \item $G$ is Cohen-Macaulay of dimension $2$;
  \item $\Delta(G)$ is connected of dimension $1$;
  \item $\gcd{(n,\bar{S})}=1$ and $\forall a,b\in \bar{S}$ we have  $b-a\notin \bar{S}$ and $n-(b+a)\notin \bar{S}$.
 \end{enumerate}
\end{Theorem}
\begin{proof}
$(1)\Leftrightarrow (2)$. Known fact. See also \cite{Ho} Corollary 4.54.

$(2)\Rightarrow (3)$. If $\Delta(G)$ is connected then there is a path in $\bar{G}\iso \Delta(G)$ connecting the vertices $0$ and $1$ (see Remark \ref{cliqueindependent}) whose vertices are
\[
 0=v_0, v_1, \ldots,v_r=1
\]
and edges
\[
 \{0, s_1\}, \{s_1, s_1+s_2\}, \ldots, \{\sum_{i=1}^{r-1}s_i, \sum_{i=1}^{r}s_i\equiv 1 \mod n\}
\]
with $s_i\in \bar{S}$. Hence there exists a relation
\[
 \sum a_i s_i\equiv 1 \mod n,\mbox{ with }a_i\in\NN, s_i\in \bar{S}.
\]
By the Euclidean algorithm  we have that $\gcd{(n,\bar{S})}=1$.  Suppose there exist $a,b\in \bar{S}$ with $b-a\in \bar{S}$. This implies $a\neq b$. We observe that $\{0,a, b\}$ is a clique in $\Delta(G)$, that is $\dim \Delta(G) \geq 2$.
In fact since $\bar{G}$ is circulant $\{0,a\}$ , $\{0,b\}$  and  $\{a, a+(b-a)=b\}$ are edges in $\bar{G}$. Now suppose that $n-(b+a)\in \bar{S}$.  We observe that $\{0,a, a+b\}$ is a clique in $\Delta(G)$. In fact since $\bar{G}$ is circulant $\{0,a\}$ , $\{a,a+b\}$  and  $\{a+b, a+b + n-(a+b)\equiv 0\}$ are edges in $\bar{G}$.  The implication $(3)\Rightarrow (2)$  follows by similar arguments. 

\end{proof}

\begin{Theorem}\label{Reg2}
 Let $G$ be a Cohen-Macaulay circulant graph $C_n(S)$ of dimension $2$. Then $\reg R/I(G)=2$.
\end{Theorem}
\begin{proof}
It is sufficient to prove  that $h_2\neq 0$ (see Remark \ref{remri} and the proof of Corollary \ref{RegEqualDepth}).
We need to compute $h_2=f_1-f_0+f_{-1}$. We observe that $f_1$ is the number of edges of $\bar{G}$. By Remark \ref{cliqueindependent} one of the two cases to study is
 \[
  \binom{n}{2}-ns,
 \]
with $h_2=\binom{n}{2}-n(s+1)+1$. The only roots $n\in \NN$ of the quadratic equation
\[
 \binom{n}{2}-n(s+1)+1=0
\]
are $1$ and $2$ with $s=0$. Absurd. The other case follows by the same argument.
\end{proof}

\begin{Theorem}\label{CMtype}
 Let $G$ be a Cohen-Macaulay circulant graph $C_n(S)$ of dimension $2$. Then its Cohen-Macaulay type is 
 \[ 
h_2=
\left \{
  \begin{tabular}{ll}
  $\binom{n}{2}-n(s+\frac{1}{2})+1$ & if $n$ is even and $\frac{n}{2}\in S$\\
  $\binom{n}{2}-n(s+1)+1$ & otherwise.  
  \end{tabular}
\right. 
\]
\end{Theorem}
\begin{proof}
 By Auslander-Buchsbaum Theorem (Theorem 1.3.3, \cite{BH2}) and since the $\depth R/I(G)=2$ we need to compute the Betti number in position $\beta_{i,j}$ when $i=n-2$. By Theorem \ref{Reg2} and the definition of Castelnuovo-Mumford regularity, the Betti numbers that are not trivially $0$ are $\beta_{n-2,j}$ in the degrees $j\in \{n-1,n\}$. We recall the Hochster's formula (see \cite{MS}, Corollary 5.1.2)
 \[
  \beta_{i,\sigma}(R/I_\Delta)=\dim_K \widetilde{H}_{|\sigma|-i-1}(\Delta_{\mid \sigma};K)
 \]
 where $\widetilde{H}(\cdot)$ is the simplicial homology and $\sigma\in \Delta$ is interpreted as squarefree degree in the minimal free resolution and it induces a restriction in $\Delta$ defined by
 \[
  \Delta_{\mid \sigma}=\{F\in \Delta: F\subseteq \sigma\}.
 \]
We observe that in the squarefree degree $\sigma$ having total degree $n-1$
\[
  \beta_{i,\sigma}=\dim_K \widetilde{H}_0(\Delta_{\mid \sigma};K)=0.
\]
In fact $\Delta\cong \bar{G}$ is connected and the same happens removing one of the vertices of the circulant graph $\bar{G}$ since circulant graphs are biconnected.
Now, if we consider the squarefree degree $\sigma$ having total degree $n$, again, by Hochster formula, we obtain
 \[
  \beta_{i,\sigma}=\dim_K \widetilde{H}_1(\Delta_{\mid \sigma};K).
 \]
In this case $\Delta_{\mid \sigma}\cong\Delta\cong\bar{G}$ and the chain complex of $\Delta$
\[
\CC: 0\to C_{1} \stackrel{\partial_1}{\to} C_0 \stackrel{\partial_0}{\to} C_{-1} \to 0,
\]
has the two homologies $\widetilde{H}_{0}=\widetilde{H}_{-1}=0$. Therefore  
 \[
\dim_K \widetilde{H}_1(\bar{G};K)=\beta_{i,\sigma}=  f_1-f_0+f_{-1}
 \]
and the assertion follows by Remark \ref{cliqueindependent}.
\end{proof}
\begin{Example}
 Let $G=C_8(\{2,3\})$ that is $\bar{S}=\{1,4\}$ (see Figure \ref{C8}). We observe that it satisfies conditions (3) of Theorem \ref{CM1}. Its Cohen-Macaulay type by Theorem \ref{CMtype} is
 \[
  \binom{8}{2}-8(2+1)+1=5.
 \]
\begin{center}

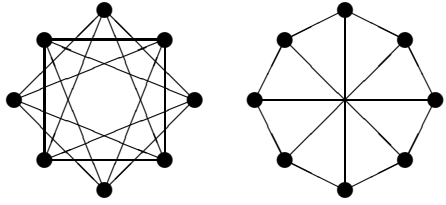
\begin{figure}

\setlength{\unitlength}{.4cm}
\begin{picture}(14,7)

\put(00,03){\circle*{.5}}
\put(01,05){\circle*{.5}}

\put(03,06){\circle*{.5}}
\put(05,05){\circle*{.5}}

\put(06,03){\circle*{.5}}
\put(05, 1){\circle*{.5}}

\put(03,0){\circle*{.5}}
\put(01,1){\circle*{.5}}

\put(00,03){\line(1,1){3}}
\put(00,03){\line(1,-1){3}}

\put(00,03){\line(5,2){5}}
\put(00,03){\line(5,-2){5}}

\put(06,03){\line(-1,-1){3}}
\put(06,03){\line(-1,1){3}}

\put(06,03){\line(-5,2){5}}
\put(06,03){\line(-5,-2){5}}

\put(01,01){\line(1,0){4}}
\put(01,01){\line(0,1){4}}

\put(05,05){\line(-1,0){4}}
\put(05,05){\line(0,-1){4}}

\put(03,00){\line(2,5){2}}
\put(03,00){\line(-2,5){2}}

\put(03,06){\line(2,-5){2}}
\put(03,06){\line(-2,-5){2}}


%
%
%



\put(08,03){\circle*{.5}}
\put(09,05){\circle*{.5}}

\put(11,06){\circle*{.5}}
\put(13,05){\circle*{.5}}

\put(14,03){\circle*{.5}}
\put(13, 1){\circle*{.5}}

\put(11,0){\circle*{.5}}
\put(09,1){\circle*{.5}}

\put(08,03){\line(1,2){1}}
\put(08,03){\line(1,-2){1}}

\put(14,03){\line(-1,-2){1}}
\put(14,03){\line(-1,2){1}}

\put(11,06){\line(2,-1){2}}
\put(11,06){\line(-2,-1){2}}

\put(11,00){\line(2,1){2}}
\put(11,00){\line(-2,1){2}}

\put(08,03){\line(1,0){6}}
\put(11,00){\line(0,1){6}}

\put(09,01){\line(1,1){4}}

\put(09,05){\line(1,-1){4}}

\end{picture}
\caption{$G=C_8(\{2,3\}$ and $C_8(\{1,4\})\cong \Delta(G)$.}\label{C8} 
 
\end{figure}

\end{center}

\end{Example}

\begin{Remark}
 We observe that the rings satisfying Theorem \ref{CMtype} are level. For a description of \textit{level} algebras see Chapter 5.4 and 5.7 of \cite{BH2}. 
\end{Remark}

\begin{Corollary}\label{Gor}
 Let $G$ be the circulant graph $C_n(S)$ with $S\subset \{1,\ldots,\left \lfloor\frac{n}{2}\right \rfloor\}$ and $s=|S|$. The following conditions are equivalent:
  \begin{enumerate}
  \item $G$ is Gorenstein of dimension $2$;
  \item $S=\{1,\ldots,\hat{i},\ldots,n\}$ and  $\gcd(n,i)=1$ with $n\geq 4$;
  \item $\Delta(G)\cong\bar{G}$ is  a $n$-gon with $n\geq 4$.
 \end{enumerate}
\end{Corollary}
\begin{proof}
 $(1)\Rightarrow (2)$. $G$ is Gorenstein if and only if $G$ is Cohen-Macaulay of type $1$. Hence by Theorem \ref{CM1} $\Delta(G)$ is connected that is $\gcd{(n,\bar{S})}=1$. Moreover by Theorem \ref{CMtype} $h_2=1$ and solving the two quadratic equations
 \[
  \binom{n}{2}-n(s+\frac{1}{2})+1=1,\, \binom{n}{2}-n(s+1)+1=1,
 \]
we obtain respectively 
\[
 n=2s+2 \mbox{ and }n=2s+3.
\]
In both cases $s=\left \lfloor\frac{n}{2}\right \rfloor-1$. Hence $\bar{S}=i$ with $\gcd(i,n)=1$ and the assertion follows.
 
 $(2)\Rightarrow (3)$. Let $\bar{S}=\{i\}$ with $\gcd(n,i)=1$. We  easily observe that the vertices
 \[
  0,i,\ldots, (n-1)i\mod n
 \]
and edges
 \[
  \{0,i\},\{i, 2i\},\ldots, \{(n-1)i, (n)i\equiv 0 \mod n\}
 \]
define a Hamiltonian cycle that is $\bar{G}$ itself.

$(3)\Rightarrow (1)$. Since $\Delta(G)$ is a simplicial $1$-sphere is Gorenstein of Krull dimension $2$ (see Corollary 5.6.5 of \cite{BH2}).
\end{proof}

We observe that in Theorem 4.1 of \cite{EMT} the Cohen-Macaulayness of the graphs described in Corollary \ref{Gor} has been studied by a different point of view.


\begin{thebibliography}{1}
\bibitem{BH0}{J. Brown, R. Hoshino},
\textit{Independence polynomials of circulants with an application to music},
\newblock Discrete Mathematics, \textbf{309}, 2009, 2292--2304.

\bibitem{BH1}{J. Brown, R. Hoshino},
\textit{Well--covered circulant graphs},
\newblock Discrete Mathematics,\textbf{311}, 2011, 244--251.

\bibitem{BH2}{W. Bruns, J. Herzog},
\textit{Cohen-Macaulay rings},
\newblock Cambridge Univ. Press, Cambridge, 1997.

\bibitem{EMT}{J. Earl, K. N. Vander Meulen, A. Van Tuyl},
\textit{Independence Complexes of Well-Covered Circulant Graphs},
\newblock Experimental Mathematics, \textbf{25}, 2016, 441--451.

\bibitem{Ei}{D. Eisenbud},
\textit{The Geometry of Syzygies},
\newblock Graduate texts in Mathematics, Springer, 2005.

\bibitem{FT}{C. A. Francisco, A. Van Tuyl},
\textit{Sequentially Cohen-Macaulay edge ideals},
\newblock  Proc. Amer. Math. Soc., \textbf{135}, 2007, 2327--2337.

\bibitem{HTYZ}{H. Haghighi, N. Terai, S. Yassemi, R. Zaare-Nahandi},
\textit{Sequentially $S_r$ simplicial complexes and sequentially $S_2$ graphs},
\newblock  Proc. Amer. Math. Soc., \textbf{139}, 2011, 1993--2005.

\bibitem{HH}{J. Herzog, T. Hibi},
\textit{Distributive lattices, bipartite graphs and Alexander duality},
\newblock J. Alg. Combin., \textbf{22}, 2005, 289--302.

\bibitem{Ho} {R. Hoshino},
\textit{Independence polynomials of circulant graphs},
\newblock PhD Thesis, Dalhouise University, 2008, 1--280.

\bibitem{MS}{E. Miller, B. Sturmfels}, 
\newblock \textit{Combinatorial Commutative Algebra}, Springer-Verlag, Berlin, 2005.

\bibitem{Mo}{A. Mousivand},
\textit{Circulant $S_2$ graphs},
\newblock preprint arXiv:1512.08141v1, 2015, 1--11.

\bibitem{RR}{Rinaldo, F. Romeo}
\textit{On the reduced Euler characteristic of independence complexes of  circulant graphs},
\newblock preprint arXiv:1706.00863, 2017, 1--12.


\bibitem{St} R.~P.~Stanley, 
 \textit{Combinatorics and Commutative Algebra, Second Edition}, 
 Birkh\"auser, Boston/Basel/Stuttgart, 1996. 

 \bibitem{MTW}{K. N. Vander Meulen, A. Van Tuyl, C. Watt},
\textit{Cohen-Macaulay circulant graphs},
\newblock Communications in Algebra,\textbf{42}, 2014, 1896--1910.

\bibitem{Va}{W. Vasconcelos},
\textit{Computational methods in commutative algebra and algebraic geometry},
\newblock Springer Science \& Business Media,\textbf{2}, 2004.

\bibitem{Vi}{R. H. Villarreal},
\textit{Cohen--Macaulay graphs}, 
\newblock Manuscripta Math.,\textbf{66} 1990, 277--293.

\bibitem{Vi0}{R. Villarreal}, 
\newblock \textit{Monomial algebras},
\newblock Marcel Dekker, New-York, 2001.

\end{thebibliography}
\end{document}